\documentclass[11pt, a4paper]{article} 
\title{Well-posedness of an asymptotic model for capillarity-driven free boundary Darcy flow in porous media in the critical Sobolev space }
\author{
	Stefano Scrobogna
	\\{\footnotesize Departamento de An\'alisis Matemático \& IMUS,  Universidad de Sevilla,  Sevilla, Espa\~na}
	\\{\footnotesize email: {\it scrobogna@us.es}}
}

\usepackage[a4paper]{geometry}
\usepackage{empheq}
\usepackage{times}
\usepackage{stmaryrd}
\usepackage{fourier}
\usepackage[T1]{fontenc}
\usepackage{amssymb, amsmath,  bm, mathrsfs}
\usepackage{amsfonts}
\usepackage[english]{babel}
\usepackage{amssymb,amsthm}
\usepackage{mathtools}
\DeclareMathAlphabet{\mathcal}{OMS}{cmsy}{m}{n}
\usepackage{hyperref}
\usepackage{cite}
\allowdisplaybreaks[1]
\usepackage{bbm}
\usepackage{xfrac}
\usepackage[utf8]{inputenc}
\usepackage[T1]{fontenc}
\usepackage{enumerate}
\usepackage{accents}

\usepackage{tikz}
\usetikzlibrary{arrows, automata, backgrounds, shadows, patterns, calc, hobby, plotmarks, shapes}
\usetikzlibrary{shapes.misc}

\tikzset{cross/.style={cross out, draw=black, minimum size=2*(#1-\pgflinewidth), inner sep=0pt, outer sep=0pt},
cross/.default={1pt}}

\newcommand{\dx}{\textnormal{d}{x}}

\newcommand{\dd}{\textnormal{d}}

\newcommand{\pare}[1]{\left( #1 \right)}
\newcommand{\angles}[1]{\left\langle #1 \right\rangle}

\newcommand{\norm}[1]{\left\| #1 \right\|}
\newcommand{\av}[1]{\left| #1 \right|}
\newcommand{\bra}[1]{\left[ #1 \right]}
\newcommand{\set}[1]{\left\{ #1 \right\}}

\newcommand{\ddt}{\frac{\textnormal{d}}{\textnormal{d}t}}
\newcommand{\pv}{\textnormal{p.v.}}

\def\comm#1#2{{\left\llbracket#1,#2\right\rrbracket}}

\newcommand{\RN}[1]{%
  \textup{\uppercase\expandafter{\romannumeral#1}}%
}
\newcommand{\nnorm}[1]{{\left\vert\kern-0.25ex\left\vert\kern-0.25ex\left\vert #1 
    \right\vert\kern-0.25ex\right\vert\kern-0.25ex\right\vert}}

\newcommand{\cC}{\mathcal{C}}

\newcommand{\cB}{\mathcal{B}}

\newcommand{\cS}{\mathcal{S}}

\newcommand{\bR}{\mathbb{R}}
\newcommand{\bN}{\mathbb{N}}

\newcommand{\bS}{\mathbb{S}}

\newcommand{\cH}{\mathcal{H}}

\newcommand{\bZ}{\mathbb{Z}}

\newcommand{\ii}{\mathbbm{i}}
\newcommand{\ee}{\mathsf{e}}

\newcommand{\hra}{\hookrightarrow}

\newcommand{\sgn}{\textnormal{sgn}}

\theoremstyle{theorem}
\newtheorem{theorem}{Theorem}[section]

\newtheorem*{theorem*}{Theorem}

\newtheorem{lemma}[theorem]{Lemma}

\theoremstyle{definition}

\numberwithin{equation}{section}

\begin{document}

\maketitle

\begin{abstract}
We prove that the quadratic approximation of the capillarity-driven free-boundary Darcy flow, derived in \cite{GS19}, is well posed in $ \dot{H}^{3/2} \pare{\bS^1}$, and globally well-posed if the initial datum is small in $ \dot{H}^{3/2} \pare{\bS^1} $.
\end{abstract}

\section{Presentation of the problem}

Fluid moving in porous media, such as sand or wood, are a common occurrence in nature. The simplest equation describing such physical phenomenon is the Darcy law 
\begin{equation}
\frac{\mu}{\beta} u=-\nabla p-\rho g \ \ee_2,
\label{eq:Darcy}
\end{equation}
where $u$, $p$, $\rho$ and $\mu$ are the velocity, pressure, density and dynamic viscosity of the fluid, respectively. The constant $\beta$ describes a property of the porous media and its known as the permeability. The term $\rho g \ee_2$ stands for the acceleration due to gravity in the direction $\ee_2 = (0,1)^\intercal $. From now on we use the renormalization $ \mu/\beta  \equiv 1 $ so that \eqref{eq:Darcy} becomes
\begin{equation*}
 u=-\nabla p-  \rho g \ \ee_2,
\end{equation*}
 Darcy law is valid for slow and viscous flows, and it was {first} derived experimentally by {Henry Darcy} (1856)  {and then} derived theoretically from the Navier-Stokes equations via homogenization (cf. \cite{Whitaker1986}). The free boundary Darcy flow, also known as Muskat problem (cf. \cite{Muskat31, Muskat32}),  is often used in order to model the dynamics of aquifiers or oil wells. When the free-interface is the graph $ \Gamma\pare{t} = \set{\pare{x, h\pare{x}} \ \left| \ x\in \bR \text{ or } \bS^1 \right. } $, which divides the space in the regions
 \begin{align*}
 \Omega^\pm\pare{t}= \set{\left. \pare{x, y}\in\bR^2 \text{ or } \bS^1\times \bR \ \right| \ y\lessgtr  h\pare{x, t} },
\end{align*}
with fluid of density
\begin{equation*}
\rho\pare{x,y,t} = \left\lbrace
\begin{array}{lll}
\rho_- & \text{ if } & \pare{x,y}\in\Omega^-\pare{t}\\
\rho_+ & \text{ if } & \pare{x,y}\in\Omega^+\pare{t}
\end{array}
\right.  ,
\end{equation*}
  (i.e. the fluid with density $ \rho_+ $ lies below and the fluid with density $ \rho_- $ lies above),   the evolution of the Muskat problem can be expressed as the contour equation
 \begin{equation} \label{eq:Muskat_completo}
 h_t = G\bra{h} \pare{\pare{\rho_+ - \rho_-} g h -\gamma \kappa} ,
 \end{equation}
 where $  G\bra{h}\psi $ is the Dirichlet-to-Neumann operator (cf. \cite{Lannes2013}), $ \kappa = \frac{h''}{\pare{1+\pare{h'}^2}^{3/2}} $ is the mean curvature of the interface, $ \gamma \geq 0 $ is the capillarity coefficient and $ g\geq 0 $ is the gravitational acceleration. \\

   The mathematical analysis of the Muskat problem has flourished in the past 20 years, see \cite{CG07, CCGO09, CCG09, CCG11, CCGS13, CGS16, AMS20, AL20, Matioc2018, Matioc2019, GGS20, CGSV17, Alazard2020, AB20} and the survey articles \cite{GL20, Gancedo2017}, but only recently due to the works of C\'ordoba \& Lazar \cite{CL18}, Gancedo \& Lazar \cite{GancedoLazar20} and Alazard \& Nguyen \cite{AN20_1, AN20_2} the problem of solvability of the Muskat problem in the critical Sobolev space $ \dot{H}^{\frac{d}{2} +1} $, where $ d $ is the dimension of the free interface, has been addressed. The three works \cite{AN20_1, AN20_2, CL18} study the stable, two phase gravity driven Muskat problem in $ \bR^2 $, i.e. $ \rho_+ > \rho_- $, $ g>0 $, $ \gamma =0 $ and the two fluids fill the two-dimensional space $ \bR^2 $. In such setting \eqref{eq:Muskat_completo} writes in the simplified form (here $ g $ is normalized to one)
   \begin{equation*}
   h_t\pare{x} =  \frac{\rho_+ - \rho_-}{2\pi} \  \pv \int_{\bR} \partial_x \arctan\pare{\frac{h\pare{x} - h\pare{x-y}}{y}}\dd y .
   \end{equation*}
   In \cite{CL18} a global well posedness result was proved for initial data $ f_0\in\dot{H}^{3/2}\cap \dot{H}^{5/2} $ with smallness assumption on $ \norm{f_0}_{\dot{H}^{3/2}} $ only, thus allowing initial data with arbitrarily large, albeit finite, slopes. In \cite{AN20_1} a global well-posedness result was proved when the initial is small w.r.t. the non-homogeneous norm
   \begin{equation*}
   \norm{u}_{H^{\frac{3}{2}, \frac{1}{3}}}^2 = \int \pare{1+\av{\xi}^2}^{3/2} \pare{\log\pare{4+\av{\xi}}}^{1/3} \av{\hat{u}\pare{\xi}}^2 \dd \xi , 
   \end{equation*}
   thus allowing initial data to have infinite slope, while \cite{CL18} and \cite{GancedoLazar20} address the problem of global solvability for initial data in $ \dot{H}^{\frac{d}{2}+1} \cap \dot{W}^{1, \infty} $ with smallness assumption in $  \dot{H}^{\frac{d}{2}+1} $ only. \\

We denote with $ \cH $ the Hilbert transform, with $ \Lambda = \cH \partial_x  $ the Calderón operator on $ \bS^1 $ or $ \bR $ and with $ \comm{A}{B}f = A\pare{Bf} - B\pare{Af} $. In \cite{GS19} the equation
 \begin{equation}\label{eq:Darcy_GC_asympt}
 f_t + g \Lambda f + \gamma \Lambda^3 f = \partial_x \comm{\cH	}{f}\pare{g\Lambda f + \gamma \Lambda^3 f}, 
 \end{equation}
 was derived, a thorough analysis of the case $ g > 0, \ \gamma \geq 0 $ was performed in \cite{GS19_2}. The equation \eqref{eq:Darcy_GC_asympt} captures the dynamics of the one-phase Muskat problem (alternatively known as the Hele-Shaw problem) subject to gravity and surface tension up to quadratic order of the nonlinearity in small amplitude number regime, we refer the reader to \cite{GS19_4, GS19_3, GS2020, BG19, CGSW19, LM2017, PM2012, EMM2012, ABN12} for further results on asymptotic models for free boundary systems. The equation we are interested to study in the present manuscript is the capillarity-driven version of \eqref{eq:Darcy_GC_asympt}, i.e. setting $ \pare{g, \gamma} = \pare{0, 1} $ we obtain the equation\footnote{The author would like to mention the very recent manuscript \cite{MM2020} in which the authors prove local solvability for arbitrary initial data in $ W^{s, p}\pare{\bR}, \ p\in\left( 1, 2\right], \ s\in \pare{  1 + \frac{1}{p}, 2} $ for the two-phases full Muskat problem, i.e. the full system of the two-phase version of \eqref{eq:Darcy_capillarity}. }
 \begin{equation}\label{eq:Darcy_capillarity}
 \left\lbrace
 \begin{aligned}
 &  f_t +\Lambda^3 f = \partial_x \comm{\cH}{f}\Lambda^3 f, \\
 & \left. f\right|_{t=0} = f_0 .
 \end{aligned}
 \right. 
 \end{equation}
 It is immediate to see that the transformation
 \begin{align}\label{eq:scaling}
 f\pare{x, t}\mapsto \frac{1}{\lambda} f \pare{\lambda x , \lambda^3 t}, 
 &&
 f_0\pare{x}\mapsto \frac{1}{\lambda} f_0 \pare{\lambda x}
 \end{align}
 where $ \lambda > 0 $ if the space domain is $ \bR $ and $ \lambda\in\bN^\star $ if the space domain is $ \bS^1 $,
 generates a one-parameter family of solutions for \eqref{eq:Darcy_capillarity}, it is hence a classical consideration in the analysis of nonlinear partial differential equations to look for solutions in functional spaces whose norm is invariant w.r.t. the transformation \eqref{eq:scaling}, a simple example of such spaces is (recall that here the space dimension is one)
 \begin{align*}
 L^\infty\pare{\bR_+; \dot{H}^{3/2}}\cap L^2\pare{\bR_+; \dot{H}^3 }, 
 &&
 L^4\pare{\bR_+ ; \dot{H}^{9/4}}. 
 \end{align*}

We prove in particular that for any $ f_0\in\dot{H}^{3/2} $ there exists a $  T = T\pare{f_0} > 0 $ and a unique solution in $ L^4\pare{\bra{0, T}; \dot{H}^{9/4}} $ of \eqref{eq:Darcy_capillarity} stemming from $ f_0 $, which is global if $ f_0 $ is small in $ \dot{H}^{3/2} $. Such result is more general than any well posedness result known, up to date, for the full Muskat problem\footnote{The author would like to point out that shortly after the publication of the preprint version of the present manuscript T. Alazard and Q.-H. Nguyen proved in \cite{AN20_3} that the gravity-driven two-phases 1D Muskat problem is locally well-posed in the \textit{nonhomogeneous} critical Sobolev space $ H^{3/2} $.  }. This is rather surprising since asymptotic models tend to be less regular compared to the full-models from which they derive (cf. \cite{ABN14, ABN19, BG20}) lacking some fine nonlinear cancellation which is present in the full system. Such result is possible thanks to a surprising commutation property of the bilinear truncation of the Dirichlet-Neumann operator. We refer the interested reader to Lemma \ref{lem:comm_estimate} for a detailed statement of the key commutation which is the fundamental tool that allows us to prove the main result of the present manuscript.

 \section{Main result and notation}

 We denote with $ C $ a positive constant whose explicit value may vary from line to line. Given a metric space $ \pare{ X, d_X } $, any  $ x_0\in X $ and $ r > 0 $ we denote with $ B_X\pare{x_0, r} $ the open ball of center $ x_0 $ and radius $ r $ w.r.t. the distance function $ d_X $. We denote with $ \ii =\sqrt{-1} $ the imaginary unit and with $ X' $ the dual space of $ X $.  \\

 From now on we consider the space domain on which \eqref{eq:Darcy_capillarity} is defined to be the one-dimensional torus $ \bS^1 $, though the computations performed in the present article can be easily adapted to the case of the one dimensional real line $ \bR $. We denote with $ \cS $ the space of Schwartz functions on $ \bS^1 $ and with $ \cS_0 $ the space of Schwartz functions with zero average. Let us denote with $ {\bf e}_n\pare{x} = e^{\ii nx}, \ n\in\bZ $ and let us consider a $ v \in \cS' $, the Fourier transform of $ v $ is defined as
\begin{align*}
\hat{v}\pare{n} = \frac{1}{2\pi} \angles{v, {\bf e}_{-n}} = \frac{1}{2\pi} \int _{-\pi}^{\pi} v\pare{x} e^{-\ii nx} \dd x, && n\in\bZ ,
\end{align*} 
since for any $ n\in\bZ $ the function $ {\bf e}_n\in\cS $ the above integral is well defined. For details we refer the interested reader to \cite{Iorio2001}. 
We define the Calder\'on operator $ \Lambda $ as the Fourier multiplier $ \widehat{\Lambda v}\pare{n} = \av{n}\hat{v}\pare{n} $, and for any $ b \in \cC \pare{ \pare{0, \infty} ; [0, \infty ) } $ we define the operator $ \widehat{b \pare{\Lambda} v}\pare{n} = b \pare{\av{n}} \hat{v}\pare{n} $. 
  We denote with
 \begin{equation*}
 \dot{H}^s = \set{v\in\cS_0' \ \left| \  \Lambda^s v\in L^2 \right. },
 \end{equation*}
 for any $ s\in\bR $. The space $ \dot{H}^s $ is endowed with the norm
\begin{equation*}
\norm{v}_{\dot{H}^s}^2 = \norm{\Lambda^s v}_{L^2}^2 = \sum_n \av{n}^{2s}\av{\hat{v}\pare{n}}^2.
\end{equation*} 
Let us remark that the homogeneous Sobolev space $ \dot{H}^s $ is the subset of the more familiar non-homogeneous Sobolev space $ {H}^s = \set{v\in\cS' \ \left| \  \pare{ 1+ \Lambda }^s v\in L^2 \right. } $ with zero average. 
  We use the abbreviated notation
 \begin{align*}
 \norm{v}_s = \norm{v}_{\dot{H}^s}, && \norm{v}_{L^p_T \dot{H}^s } = \norm{v}_{L^p\pare{\bra{0, T}; \dot{H}^s}} ,  && \norm{v}_{L^p \dot{H}^s } =  \norm{v}_{L^p_\infty \dot{H}^s }, && T\in\left( 0, \infty\right], \ p\in\bra{1, \infty}. 
 \end{align*}

 The main result we prove is the following one
 \begin{theorem}\label{thm:main}
 Given $ f_0 \in  \dot{H}^{3/2} $ there exists a $ T = T\pare{ f_0 } > 0 $ such that the system \eqref{eq:Darcy_capillarity} has a unique solution in the space $ L^4\pare{\bra{0, T} ; \dot{H}^{9/4}} $, which, in addition, belongs to the space
\begin{equation*}
f\in\cC\pare{\bra{0, T}; \dot{H}^{3/2}}\cap L^2\pare{\bra{0, T}; \dot{H}^3}. 
\end{equation*} 
  There exists an $ \varepsilon_0 > 0 $ such that if
 \begin{equation*}
\norm{f_0}_{3/2}\leq \varepsilon_0,  
 \end{equation*}
 then $ T=\infty $ and the solution is global.
 \end{theorem}

 In Section \ref{sec:preliminaries} we introduce some preliminary results which we use in Section \ref{sec:proof} in order to prove Theorem \ref{thm:main}.

 \section{Preliminaries } \label{sec:preliminaries}

Let us state the \textit{Minkowsky integral inequality} (cf. \cite[Appendix A]{Stein1970}) : let us consider $ \pare{S_1, \mu_1} $ and $ \pare{S_2, \mu_2} $ two $ \sigma $--finite measure spaces and let $ f : S_1\times S_2 \to \bR $ be measurable, $ p\in[1, \infty) $, then the following inequality holds true:
\begin{equation}\label{eq:Minkowsky_integral}
 \left[\int _{S_{2}}\left|\int _{S_{1}}f(x,y)\,\mu _{1}(\mathrm {d} x)\right|^{p}\mu _{2}(\mathrm {d} y)\right]^{\frac {1}{p}}\leqslant \int _{S_{1}}\left(\int _{S_{2}}|f(x,y)|^{p}\,\mu _{2}(\mathrm {d} y)\right)^{\frac {1}{p}}\mu _{1}(\mathrm {d} x).
\end{equation} 
The above equality holds as well when $ p=\infty $ with obvious modifications.\\

 The proof of Theorem \ref{thm:main} relies on a fixed point argument, in particular the fixed point theorem we rely on is the following one (see \cite[Lemma 5.5 p. 207]{BCDbook})
 \begin{lemma}\label{lem:fixed_point}
 Let $ X $ be a Banach space, $ \cB $ a continuous bilinear map form $ X \times X $ to $ X $  and $ r > 0 $ such that
 \begin{align*}
 r < \frac{1}{4\norm{\cB} } , 
 &&
 \norm{\cB} = \sup _{u,v\in B_X\pare{0, 1}} \norm{\cB\pare{u, v}}_X . 
 \end{align*}
 For any $ x_0$ in the ball $ B_X\pare{0, r } = \set{y\in X \ : \ \norm{y}_X  < r} $, there exists a unique $ x\in B_X\pare{0, 2r} $ such that
 \begin{equation*}
 x= x_0 + \cB\pare{x,x}.
 \end{equation*}
 \end{lemma}

The next result we need is a particular commutation property which is specific to the nonlinearity of the equation \eqref{eq:Darcy_capillarity}
\begin{lemma}\label{lem:comm_estimate}
Let $ s, \sigma\geq 0 , \ \alpha \in\bra{0, \sigma} $ and $ \phi\in \dot{H}^{s + \frac{1}{4} +\alpha }, \ \psi \in \dot{H}^{\sigma+\frac{1}{4} - \alpha} $ then we have that
\begin{equation*}
\norm{\Lambda^s \pare{ \comm{\cH}{\phi}\Lambda^\sigma \psi }}_{L^2} \leq C \norm{\phi}_{s+\frac{1}{4} + \alpha}\norm{\psi}_{\sigma +\frac{1}{4} -\alpha}. 
\end{equation*}
\end{lemma}
\begin{proof}
Let us remark that
\begin{equation*}
\widehat{\comm{\cH}{\phi}\Lambda^\sigma \psi }\pare{n} =  \ii  \sum_{k\in\bZ} \pare{-\sgn\pare{n} + \sgn\pare{n-k}} \av{n-k}^\sigma \hat{\phi}\pare{k}\hat{\psi}\pare{n-k}. 
\end{equation*}
Now we have that
\begin{align*}
-\sgn\pare{n} + \sgn\pare{n-k} \neq 0 && \Leftrightarrow && \pare{ n> 0 \wedge n-k < 0 } \vee \pare{n<0 \wedge n-k > 0} && \Leftrightarrow && 0 < \av{n} < \av{k} ,  
\end{align*}
as a consequence we obtain that
\begin{align}\label{eq:monotonicity_Fourier_modes}
\av{n} < \av{k}, && \text{ and } && \av{n-k} < \av{k}. 
\end{align}
Using the monotonicity property \eqref{eq:monotonicity_Fourier_modes} we obtain that
\begin{align*}
\norm{\Lambda^s \pare{ \comm{\cH}{\phi}\Lambda^\sigma \psi }}_{L^2} ^2 \leq & \ C \sum_n \av{n}^{2s} \av{
\sum_{k}  \av{n-k}^\sigma \av{\hat{\phi}\pare{k}}\av{\hat{\psi}\pare{n-k}}
}^2 , \\
 \leq & \ C \sum_n \av{
\sum_{k} \av{k}^{s+\alpha}  \av{n-k}^{\sigma -\alpha}\av{\hat{\phi}\pare{k}}\av{\hat{\psi}\pare{n-k}}
}^2 , \\
\leq & \ C \norm{\Lambda^{s+\alpha}\Phi \ \Lambda^{\sigma-\alpha}\Psi}_{L^2}^2 , 
\end{align*}
where
\begin{align*}
\hat{\Phi}\pare{n} = \av{\hat{\phi}\pare{n}}, && \hat{\Psi}\pare{n} = \av{\hat{\psi}\pare{n}}.
\end{align*}
We use now the H\"older inequality and the embedding $ \dot{H}^{1/4}\hra L^4 $
\begin{equation*}
\norm{\Lambda^{s+\alpha}\Phi \ \Lambda^{\sigma-\alpha}\Psi}_{L^2} \leq C  \norm{\Phi}_{s+\frac{1}{4} + \alpha}\norm{\Psi}_{\sigma +\frac{1}{4} -\alpha} = C  \norm{\phi}_{s+\frac{1}{4} + \alpha}\norm{\psi}_{\sigma +\frac{1}{4} -\alpha}, 
\end{equation*}
concluding the proof.
\end{proof}

Let $ \phi, \psi $ be given functions, let us denote with $ U=U\pare{\phi, \psi} $ the solution to the linear fractional-diffusion equation with forcing
\begin{equation}\label{eq:U}
\left\lbrace
\begin{aligned}
& U\pare{\phi, \psi}_t + \Lambda^3 U\pare{\phi, \psi} = \partial_x \comm{\cH}{\phi}\Lambda^3\psi, \\
& \left. U\pare{\phi, \psi} \right|_{t=0} = 0.
\end{aligned}
\right. 
\end{equation}

\begin{lemma}\label{lem:cont_bilinear_map}
For any $ \phi\in L^4_T \dot{H}^{s + \frac{3}{4}}, \ s\geq 1/2 $ and  $ \psi\in L^4_T\dot{H}^{9/4} $ the following inequality holds true
\begin{equation*}
\norm{U\pare{\phi, \psi}}_{L^4_T \dot{H}^{s+\frac{3}{4}}} \leq C \norm{\phi}_{L^4_T \dot{H}^{s+\frac{3}{4}}} \norm{\psi}_{L^4_T \dot{H}^{9/4}}. 
\end{equation*}
\end{lemma}
\begin{proof}
We write $ U $ instead of $ U\pare{\phi, \psi} $ for sake of simplicity. A standard energy estimate combined with Lemma \ref{lem:comm_estimate} give that
\begin{align*}
\frac{1}{2}\ddt \norm{U}_{s}^2 + \norm{U}_{s+\frac{3}{2}}^2 = & \ \int \Lambda^s\pare{\partial_x\comm{\cH}{\phi}\Lambda^3 \psi} \Lambda^s U \dx, \\
\leq & \ \frac{1}{2}\norm{U}_{s+\frac{3}{2}}^2 +  C \norm{\Lambda^{s-\frac{1}{2}}\pare{ \comm{\cH}{\phi}\Lambda^3\psi }}_{L^2}^2 , \\
\leq & \ \frac{1}{2}\norm{U}_{s+\frac{3}{2}}^2 +  C\norm{\phi}_{s+\frac{3}{4}}^2\norm{\psi}_{9/4}^2 , 
\end{align*}
hence, integrating in time, we obtain the bound
\begin{equation*}
\norm{U}_{L^\infty_T \dot{H}^s}^2 + \norm{U}_{L^2_T\dot{H}^{s+\frac{3}{2}}}^2 \leq  C \norm{\phi}_{L^4_T \dot{H}^{s+\frac{3}{4}}}^2  \norm{\psi}_{L^4_T \dot{H}^{9/4}}^2. 
\end{equation*} 
The interpolation estimate
\begin{equation*}
\norm{U }_{L^4_T \dot{H}^{s+\frac{3}{4}}} \leq \norm{U}_{L^\infty_T \dot{H}^s}^{1/2}  \norm{U}_{L^2_T\dot{H}^{s+\frac{3}{2}}}^{1/2 } \leq C \pare{ \norm{U}_{L^\infty_T \dot{H}^s} + \norm{U}_{L^2_T\dot{H}^{s+\frac{3}{2}}} }, 
\end{equation*}
combined with the Young inequality $ ab\leq \frac{1}{2}\pare{a^2 + b^2} $ provides the desired control. 
\end{proof}

\section{Proof of Theorem \ref{thm:main}}	\label{sec:proof}

Solving the Cauchy problem \eqref{eq:Darcy_capillarity} is equivalent to find a solution to the integral equation
\begin{align}\label{eq:Duhamel_formulation}
f\pare{x,t} = e^{-t\Lambda^3}f_0\pare{x} + \int_0^t e^{-\pare{t-t'}\Lambda^3}\partial_x \comm{\cH}{f}\Lambda^3 f\pare{x, t'}\dd t', && \pare{x, t}\in\bS^1\times \bra{0, T} ,
\end{align}
we  hence prove that the map
\begin{equation}\label{eq:functional_equality}
f \longmapsto e^{-t\Lambda^3} f_0 + U\pare{f, f}, 
\end{equation}
where $ U\pare{f,f }$ is defined as the solution of \eqref{eq:U}, has a unique fixed point in the space $ L^4_T \dot{H}^{9/4} $. Indeed for any $ T\in\left(0, \infty\right] $ and $ \phi, \psi \in  L^4_T \dot{H}^{9/4} $ we invoke Lemma \ref{lem:cont_bilinear_map} obtaining that
\begin{equation*}
\norm{U\pare{\phi, \psi}}_{L^4_T \dot{H}^{9/4}} \leq C \norm{\phi}_{L^4_T \dot{H}^{9/4}} \norm{\psi}_{L^4_T \dot{H}^{9/4}},
\end{equation*}
 thus proving that $ U $ is a continuous bilinear map form $ \pare{ L^4_T \dot{H}^{9/4}}^2 $ onto $  L^4_T \dot{H}^{9/4} $. What remains to prove in order to apply Lemma \ref{lem:fixed_point} to the functional equality \eqref{eq:functional_equality} is that for any $ f_0 \in \dot{H}^{3/2} $ here exists a $ T = T\pare{ f_{0} }\in\left(0, \infty\right] $ such that $ \norm{e^{-\bullet \Lambda^3}f_0}_{L^4_T \dot{H}^{9/4}}  $ can be made arbitrarily small. We treat at first the case of small initial datum: we use \eqref{eq:Minkowsky_integral} in order to compute
 \begin{equation}\label{eq:est_hat_propagator}
 \begin{aligned}
 \norm{e^{-\bullet \Lambda^3}f_0}_{L^4_T \dot{H}^{9/4}} = & \ \pare{\int_0^T \pare{\sum_n e^{-2t\av{n}^3}\av{n}^{9/2}\av{\hat{f}_0\pare{n}}^2 }^2\dd t}^{1/4}, \\
 \leq & \pare{\sum_n \pare{\int_0^Te^{-4t\av{n}^3}\av{n}^9\av{\hat{f}_0\pare{n}}^4 \dd t}^{1/2}}^{1/2}, \\
 \leq & \frac{1}{\sqrt{2}} \norm{f_0}_{3/2}.
 \end{aligned}
\end{equation}  
The bound derived above is independent of $ T > 0 $, hence
\begin{equation*}
\norm{e^{-\bullet \Lambda^3}f_0}_{L^4 \dot{H}^{9/4}} \leq \frac{1}{\sqrt{2}} \norm{f_0}_{3/2}.
\end{equation*}
The above inequality proves that the application $ f_0 \mapsto  e^{-\bullet \Lambda^3}f_0 $ is continuous in zero as an application from $ \dot{H}^{3/2} $ to $ L^4 \dot{H}^{9/4} $, hence there exists a $ \varepsilon_0 > 0 $ s.t. if $ \norm{f_0}_{3/2}\leq \varepsilon_0 $ the conditions of Lemma \ref{lem:fixed_point} are satisfied and there exists a unique solution of \eqref{eq:Darcy_capillarity} in the space $ L^4 \pare{\bR_+;  \dot{H}^{9/4} } $. \\

We consider now the case of large initial data in $ \dot{H}^{3/2} $. Let us set $ \rho > 0 $ and define $ f_0 =\bar{f}_0 + \underline{f_0} $ where $ \hat{\bar{f}}_0\pare{n} = \mathbf{1}_{\set{\av{n}>\rho}}\pare{n}\hat{f}_0 \pare{n} $. Since $ \bar{f}_0 $ is localized on high-frequencies,  we exploit \eqref{eq:est_hat_propagator} in order to obtain that 
\begin{equation*}
\norm{e^{-\bullet \Lambda^3}\bar{f}_0}_{L^4_T \dot{H}^{9/4}} \leq C\norm{ \bar{f}_0 }_{3/2} = o\pare{1}\quad \text{as} \quad \rho\to\infty,
\end{equation*}
and thus can be made arbitrarily small letting $ \rho = \rho\pare{ f_0 } $ be large enough. Next we use the localization in the Fourier space in order to argue that
\begin{align*}
\norm{e^{-\bullet \Lambda^3}\underline{f_0}}_{L^4_T \dot{H}^{9/4}} \leq \rho^{3/4} \norm{e^{-\bullet \Lambda^3}\underline{f_0}}_{L^4_T \dot{H}^{3/2}} , 
\end{align*}
while using again \eqref{eq:Minkowsky_integral} we obtain that
\begin{align*}
\norm{e^{-\bullet \Lambda^3}\underline{f_0}}_{L^4_T \dot{H}^{3/2}} \leq & \ \pare{\sum_n \av{n}^3 \sqrt{\frac{1-e^{-4T\av{n}^3}}{4\av{n}^3}} \av{\underline{\hat{f}_0}\pare{n}}^2}^{1/2} \leq C \sqrt[4]{T} \norm{f_0}_{3/2}, 
\end{align*}
which in turn implies that
\begin{equation*}
\norm{e^{-\bullet \Lambda^3}\underline{f_0} }_{L^4_T \dot{H}^{9/4}} \leq C \sqrt[4]{\rho^3 T}  \norm{f_0}_{3/2},
\end{equation*}
thus proving that if $ T \ll 1/\rho^3 $ we can again apply Lemma \ref{lem:fixed_point} proving the existence part of the statement of Theorem \ref{thm:main} for arbitrary data in $ \dot{H}^{3/2} $. The fact that
\begin{equation*}
f\in L^\infty\pare{\bra{0, T}; \dot{H}^{3/2}}\cap L^2\pare{\bra{0, T}; \dot{H}^3}, 
\end{equation*}
follows by a $ \dot{H}^{3/2} $ energy estimate on the equation \eqref{eq:Darcy_capillarity}, while using the Duhamel formulation \eqref{eq:Duhamel_formulation} we obtain that, fixed $ n\in\bZ $, the application $ t\mapsto \hat{f}\pare{n, t} $ is continuous over $ \bra{0, T} $, the Lebesgue dominated convergence allows us to conclude that $ f\in\cC\pare{\bra{0, T}; \dot{H}^{3/2}} $, concluding the proof of Theorem \ref{thm:main}. 
\hfill $ \Box $

\section*{Acknowledgments}
The research of S.S. is supported by the European Research Council through the Starting Grant project H2020-EU.1.1.-639227 FLUID-INTERFACE.

\begin{footnotesize}

\end{footnotesize}

\end{document}